\documentclass{amsart}

\usepackage{graphicx}
\usepackage{pictex}

\newtheorem{defn}{Definition}

\newtheorem{pro}{Proposition}
\newtheorem{algo}{Algorithm}
\newtheorem{rem}{Remark}

\def\csch{\mathop{\mathrm{csch}}\nolimits}
\def\Li{\mathop{\mathrm{Li}}\nolimits}

\begin{document}

\title{The series limit of $\sum_k 1/[k \log k (\log \log k)^2]$}

\author{Richard J. Mathar}
\urladdr{https://www.mpia-hd.mpg.de/~mathar}
\email{mathar@mpia.de}
\address{Max-Planck Institute of Astronomy, K\"onigstuhl 17, 69117 Heidelberg, Germany}


\date{\today}

\begin{abstract}
The slowly converging series
$\sum_{k=3}^\infty 1/[k\log k(\log \log k)^\alpha]$
is evaluated to $\approx 38.4067680928$ at $\alpha=2$.
After some initial terms, the infinite tail of the sum
is replaced by the integral of the associated interpolating function,
which is available in simple analytic form.
Biases that originate from the difference between the smooth
area under the function and the corresponding Riemann sum
are corrected by standard means.
The cases $\alpha=3$ and $\alpha=4$ are computed in the same manner.
\end{abstract}

\maketitle

\section{Aim and Scope}

We aim at a precise numerical evaluation of the
series limit of
\begin{equation}
C\equiv \sum_{k=3}^\infty \frac{1}{k\log k(\log \log k)^2}
,
\label{eq.Cdef}
\end{equation}
which has been
estimated at $C\approx 38.43$
in the CRC tables \cite[p.\ 42]{ZwillingerCRC} and $\approx 38.406768$
by Boas \cite{BoasAMM84}.
The presence of the square of the double logarithm in the denominator
is just sufficient to achieve convergence; direct numerical summation
is a futile strategy to estimate the series limit.
This work addresses how the family of closely related
\begin{defn}
\begin{equation}
C^{(\alpha)}\equiv \sum_{k=3}^\infty \frac{1}{k\log k(\log \log k)^\alpha}
\label{eq.Calphadef}
\end{equation}
\end{defn}
is calculated by standard Euler-Maclaurin methods of numerical
analysis for $\alpha=2$ to $\alpha=4$.

\section{Numerical strategy}
\subsection{Romberg Integration}
Some initial terms up to $k=N$ are summed directly.
Fig.\ \ref{fig.str}  demonstrates the methodology for all larger indices:
Starting at $k=N+1$, this could
be made precise by adding the areas of the rectangular boxes, but they
are substituted by the area $I_N^{(\alpha)}$ under the interpolating smooth function.
In consequence, a curvature
correction is needed since the areas added and omitted by this integral
at the top of each box---dotted areas in the inset illustrating
the case of one of these---are of different size.

\begin{figure}
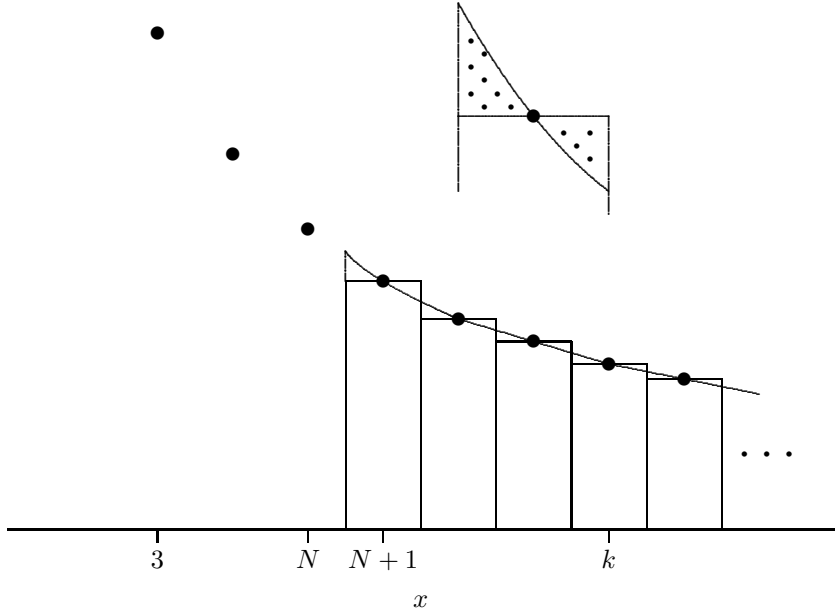

\beginpicture
\setcoordinatesystem units <1mm,1mm>
\setplotarea x from 10 to 120, y from 10 to 75
\axis bottom shiftedto x=0 label $x$ ticks withvalues 3 $N$ $N+1$ $k$ / at 30 50 60 90 / /
\put {\circle*{5}} [Bl] at 30 66
\put {\circle*{5}} [Bl] at 40 50
\put {\circle*{5}} [Bl] at 50 40
\put {\circle*{5}} [Bl] at 60 33
\put {\circle*{5}} [Bl] at 70 28
\put {\circle*{5}} [Bl] at 80 25
\put {\circle*{5}} [Bl] at 90 22
\put {\circle*{5}} [Bl] at 100 20
\put {\circle*{2}} [Bl] at 108 10
\put {\circle*{2}} [Bl] at 111 10
\put {\circle*{2}} [Bl] at 114 10
\setquadratic \plot 55 37 60 33 70 28 80 25 90 22 100 20 110 18 /

\setshadesymbol <1mm,1mm,1mm,1mm> ( {\circle*{2}} ) 
\put {\circle*{5}} [Bl] at 80 55
\vshade 70 55 68    75 55 63   80 55 55 /
\plot 70 70    80 55   90 45 /
\vshade 80 55 55   85 50 55  90 40 55 /
\setlinear \plot 70 45 70 55 90 55 90 42 / \plot 70 55 70 70  /

\sethistograms
\plot 55 0 65 33 75 28 85 25 95 22 105 20 /
\setlinear \plot 55 33 55 37 /
\endpicture
\caption{The summation of $C^{(\alpha)}$ accumulates the ordinate values
of a set of discrete points (bold dots). The integral
under the smooth curve stretching from $x=N+\frac{1}{2}$ to infinity
is
given by (\ref{eq.In}).
}
\label{fig.str}
\end{figure}

\begin{pro}
\begin{equation}
I_N^{(\alpha)}=\int_{x=N+1/2}^\infty \frac{dx}{x\log x(\log \log x)^\alpha}
=
\frac{1}{(\alpha-1)\left[\log \log (N+\frac{1}{2})\right]^{\alpha-1}},
\quad \alpha=2,3,\ldots.
\label{eq.In}
\end{equation}
\end{pro}
\begin{proof}
This follows from the substitution
\begin{equation}
\log x=\frac{1}{z};\quad 
x=e^{1/z};\quad
dx=-\frac{dz}{z^2}e^{1/z};\quad
\frac{dx}{x}=-\frac{dz}{z^2}
\end{equation}
via \cite[2.721.2]{GR}
\begin{gather}
I_N^{(\alpha)}=\int_{z=0}^{1/\log(N+1/2)} \frac{dz}{z(\log z)^\alpha}
=
-\left.\frac{1}{(1-\alpha)\log^{\alpha-1} z}\right|_{0}^{1/\log(N+1/2)}
\end{gather}
\end{proof}
\begin{rem}
The Cauchy-Maclaurin integral test plus the finiteness of
this integral proof the convergence of (\ref{eq.Calphadef}).
\end{rem}
\begin{rem}
(\ref{eq.In}) estimates that an absolute accuracy of $\Delta C^{(2)}\approx 1$ is
reached after direct summation of $N\approx e^{e^{1/\Delta C}}\approx 15$ terms, an accuracy
of $\Delta C^{(2)}\approx 0.1$ only after $N\approx
9\times 10^{9565}$ terms.
\end{rem}

The half-infinite interval is cut into abscissa sections of unit width
centered at integer $k$,
\begin{equation}
I_N^{(\alpha)}=\sum_{k=N+1}^\infty A_k^{(\alpha)}.
\end{equation}

The curvature correction is estimated through a
Taylor series expansion around the mid-point $x=k$ in each unit interval,
\begin{eqnarray}
A_k^{(\alpha)}
&=&\int_{k-1/2}^{k+1/2} \frac{dx}{x\log x(\log\log x)^\alpha}
\\
&=&
\int_{k-1/2}^{k+1/2} \Big[
\frac{1}{k\log k(\log \log k)^\alpha}+\cdots\frac{(x-k)^1}{1!}
\nonumber
\\
&&\quad +\frac{d^2}{dx^2}\frac{1}{x\log x(\log \log x)^\alpha}\frac{(x-k)^2}{2!}
+\cdots \frac{(x-k)^3}{3!}
+\cdots
\Big]dx
.
\end{eqnarray}
The integrals over terms with odd powers of $(x-k)$ do not
contribute due to the symmetry of the limits;
only terms $\propto (x-k)^{2s}$ remain,
\begin{equation}
A_k^{(\alpha)}
=
\frac{1}{k\log k(\log \log k)^\alpha}+\sum_{s=1}^\infty
\int_{k-1/2}^{k+1/2}
\frac{d^{2s}}{dx^{2s}}\frac{1}{x\log x(\log \log x)^\alpha}_{\mid k}\frac{(x-k)^{2s}}{(2s)!}
dx
.
\end{equation}
Inserting the elementary \cite[2.01.1]{GR}
\begin{equation}
\int_{k-1/2}^{k+1/2} (x-k)^{2s}dx
=\frac{1}{4^s(2s+1)}
\end{equation}
turns this into
\cite[Thrm. 1]{BredeITSF17}
\begin{equation}
A_k^{(\alpha)}=
\frac{1}{k\log k(\log \log k)^\alpha}
+\sum_{s=1}^\infty
\frac{1}{4^s(2s+1)!}
\frac{d^{2s}}{dx^{2s}}\frac{1}{x\log x(\log \log x)^\alpha}
_{\mid x=k}
.
\label{eq.Ck}
\end{equation}
The first $N$ terms of (\ref{eq.Calphadef}) are summed as they stand,
and the terms from $N+1$ on are rephrased according to (\ref{eq.Ck}),
\begin{gather}
C^{(\alpha)}=
\sum_{k=3}^N
\frac{1}{k\log k(\log \log k)^\alpha}
+
\sum_{k=N+1}^\infty
\frac{1}{k\log k(\log \log k)^\alpha}
\\
=
\sum_{k=3}^N
\frac{1}{k\log k(\log \log k)^\alpha}
+\sum_{k=N+1}^\infty
\left(
A_k^{(\alpha)}-\sum_{s=1}^\infty
\frac{1}{4^s(2s+1)!}
\frac{d^{2s}}{dx^{2s}}\frac{1}{x\log x(\log \log x)^\alpha}
\right)
\nonumber
\end{gather}
to yield
\begin{algo} (Romberg)
\begin{gather}
C^{(\alpha)}
=
\sum_{k=3}^N
\frac{1}{k\log k(\log \log k)^\alpha}
+I_N^{(\alpha)}
-
\sum_{s=1}^\infty
\frac{1}{4^s(2s+1)!}
\sum_{k=N+1}^\infty
\frac{d^{2s}}{dx^{2s}}\frac{1}{x\log x(\log \log x)^\alpha}_{\mid k}
.
\label{eq.final}
\end{gather}
\end{algo}

The derivatives $d^{2s}/dx^{2s}$
share a common
format, which suggests the
notational shortcut
\begin{defn} (Atoms of Derivatives)
\begin{equation}
g(n,l,L)\equiv \frac{1}{x^n (\log x)^l (\log \log x)^L}
.
\end{equation}
\end{defn}
The derivatives for $s=1$ or $s=2$ then read
\begin{gather}
\frac{d^2}{dx^2}g(1,1,\alpha)
=
2g(3,1,\alpha)
+3g(3,2,\alpha)
+3\alpha g(3,2,1+\alpha)
\label{eq.d22}
+2g(3,3,\alpha)\\
+3\alpha g(3,3,1+\alpha)
+\alpha(1+\alpha) g(3,3,1+\alpha)
\nonumber
,
\end{gather}
and
\begin{gather}
\frac{d^4}{dx^4}g(1,1,\alpha)
=
24g(5,1,\alpha)
+50g(5,2,\alpha)
+70g(5,3,\alpha)
+60g(5,4,\alpha)
\label{eq.d43}
\\
+24g(5,5,\alpha)
+50\alpha g(5,2,1+\alpha)
+105\alpha g(5,3,1+\alpha)
+35\alpha(1+\alpha) g(5,3,2+\alpha)
\nonumber
\\
+110\alpha g(5,4,1+\alpha)
+60\alpha(1+\alpha) g(5,4,2+\alpha)
+10\alpha(2+3\alpha+\alpha^2)g(5,4,3+\alpha)
\nonumber
\\
+50\alpha g(5,5,1+\alpha)
+35\alpha(1+\alpha) g(5,5,2+\alpha)
+10\alpha(2+3\alpha+\alpha^2)g(5,5,3+\alpha)
\nonumber
\\
+\alpha(6+11\alpha+6\alpha^2+\alpha^3)g(5,5,4+\alpha)
,
\nonumber
\end{gather}
for example.

All terms are positive---so all these curvature corrections reduce
the integral estimator $I_N^{(\alpha)}$ in (\ref{eq.final}), as expected
from Fig.\ \ref{fig.str}
for convex series of points.
The terms of order $s$ have a factor $x^{2s+1}$ in their denominator---they
converge quicker than the original series.

The properties of (\ref{eq.final})
for different switch-over values $N$ are  illustrated in Table \ref{tab.C2}.
Its column $C^{(\alpha)}$ is the value of (\ref{eq.final}) after
replacing the infinite upper limit in $\sum_k$
by $\hat k$, and by including the corrections
of orders $s=1$, $2$ and $3$.
The columns headed $s=1$ or $s=2$ show
$[4^s(2s+1)!]^{-1} \sum_{k=N+1}^{\hat k} d^{2s}/dx^{2s} g(1,1,\alpha)$
to give an impression of the cumulative magnitude of the curvature corrections.
\begin{rem}
The numbers in the column $s=2$ appear to be constant as a function
of $\hat k$ at 
the precision shown, because the contributions from the 4th
derivatives are $\propto 1/k^5$---as argued above---, so their partial
sums have already converged at $\hat k=800$.
\end{rem}

\begin{table}
\caption{Convergence of (\ref{eq.final}) at $\alpha=2$.
}
\begin{tabular}{|l|l|l|l|l|}
\hline
$N$ & $\hat k$ & $C^{(2)}$ & $s=1$ & $s=2$ \\
\hline
20&400&38.4067681111183854426&0.0000517608816&0.0000000214938\\
20&800&38.4067680963437039923&0.0000517756562&0.0000000214939\\
20&1600&38.4067680935234520571&0.0000517784765&0.0000000214939\\
20&3200&38.4067680929653951654&0.0000517790345&0.0000000214939\\
20&6400&38.4067680928518229268&0.0000517791481&0.0000000214939\\
\hline40&400&38.4067681111183785953&0.0000067243761&0.0000000005800\\
40&800&38.4067680963436971450&0.0000067391508&0.0000000005800\\
40&1600&38.4067680935234452098&0.0000067419711&0.0000000005800\\
40&3200&38.4067680929653883181&0.0000067425291&0.0000000005800\\
40&6400&38.4067680928518160795&0.0000067426427&0.0000000005800\\
\hline80&400&38.4067681111183785868&0.0000010106346&0.0000000000196\\
80&800&38.4067680963436971365&0.0000010254093&0.0000000000196\\
80&1600&38.4067680935234452013&0.0000010282295&0.0000000000196\\
80&3200&38.4067680929653883095&0.0000010287876&0.0000000000196\\
80&6400&38.4067680928518160710&0.0000010289011&0.0000000000196\\
\hline
\end{tabular}
\label{tab.C2}
\end{table}

\subsection{Euler-Maclaurin}
Recursive use of this technique for the sums originating from
the curvature approximations generates an Euler-Maclaurin formula \cite{MartensenZAMM85}
\begin{algo} (Centered Euler-Maclaurin)
\begin{gather}
C^{(\alpha)}=
\sum_{k=3}^N
\frac{1}{k\log k(\log \log k)^\alpha}
+
A_k^{(\alpha)}
+\sum_{s=1}^\infty \frac{1}{2^{2s-1}}\beta(s)\frac{d^{2s-1}}{dx^{2s-1}}g(1,1,\alpha)
.
\label{eq.eul}
\end{gather}
\end{algo}
Here,
\begin{equation}
\frac{\beta(s)}{2^{2s-1}} = 
\sum_{\pi(s)=[1^{m_1},2^{m_2},\cdots ,s^{m_s}]}
\prod_j \left(-\frac{1}{4^j(2j+1)!}\right)^{m_j}
\end{equation}
is a sum over all ordered partitions (compositions) of $s$ that
accumulate the different paths of insertions of curvatures.
We first note that the factor $1/2^{2s}$ on the left hand side
cancels with the product $\prod_j 1/(4^{jm_j})$ on the right hand side.
Reverse
application of Vella's variant of Fa\`a di Bruno's formula \cite[(2)]{VellaINT8}
transforms this into \cite{GouldAMM70}
\begin{gather}
\beta(s)= (2^{2s-1}-1)\frac{B_{2s}}{(2s)!}
\end{gather}
in terms of signed Bernoulli numbers $B$ \cite[Ch.\ 23]{AS}
with generating function \cite[4.5.65]{AS}
\begin{equation}
1-z\csch z = 2\sum_{s=1}^\infty \beta(s)z^{2s}
.
\end{equation}
With this approach, equations (\ref{eq.d22})--(\ref{eq.d43}) are not needed
and are replaced by derivatives of odd order in (\ref{eq.eul})---to be evaluated
at $x=N+\frac{1}{2}$. The dominant orders are
\begin{gather}
\frac{d}{dx}g(1,1,\alpha)
=
-g(2,1,\alpha)
-g(2,2,\alpha)
-\alpha g(2,2,1+\alpha);
\end{gather}
\begin{gather}
\frac{d^3}{dx^3}g(1,1,\alpha)
=
-6g(4,1,\alpha)
-11g(4,2,\alpha)
-11\alpha g(4,2,1+\alpha)
-12g(4,3,\alpha)
\\
-18\alpha g(4,3,1+\alpha)
-6\alpha(1+\alpha) g(4,3,2+\alpha)
-6g(4,4,\alpha)
-11\alpha g(4,4,1+\alpha)
\nonumber
\\
-6\alpha(1+\alpha) g(4,4,2+\alpha)
-\alpha(2+3\alpha+\alpha^2) g(4,4,3+\alpha).
\nonumber
\end{gather}

Numerical examples are shown in Table \ref{tab.Eul1}
as a function of an upper limit $\hat s$ introduced to
the $s$-sum in (\ref{eq.eul}).
\begin{table}
\caption{Convergence of (\ref{eq.eul}). The rows $\hat s=5$ include
derivatives up to $(d^9/dx^9)g(1,1,\alpha)$.
}
\begin{tabular}{|l|l|r|r|r|}
\hline
$\alpha$ & $\hat s$ & $N=20$ & $N=40$ & $N=80$ \\
\hline
2 & 0 & 38.406819893505282& 38.406774836074573& 38.406769121772549\\ 
2 & 1 & 38.406768042600461& 38.406768091468035& 38.406768092776058\\ 
2 & 2 & 38.406768092940813& 38.406768092822471& 38.406768092821792\\ 
2 & 3 & 38.406768092821262& 38.406768092821786& 38.406768092821786\\ 
2 & 4 & 38.406768092821790& 38.406768092821786& 38.406768092821786\\ 
2 & 5 & 38.406768092821786& 38.406768092821786& 38.406768092821786\\ 

\hline3 & 0 & 372.804546001966145& 372.804497663931222& 372.804492644908823\\ 
3 & 1 & 372.804491815956332& 372.804491878037966& 372.804491879344629\\ 
3 & 2 & 372.804491879555600& 372.804491879383635& 372.804491879382884\\ 
3 & 3 & 372.804491879382026& 372.804491879382878& 372.804491879382879\\ 
3 & 4 & 372.804491879382886& 372.804491879382879& 372.804491879382879\\ 
3 & 5 & 372.804491879382879& 372.804491879382879& 372.804491879382879\\ 

\hline4 & 0 & 3898.687393982966873& 3898.687343358247550& 3898.687339020151279\\ 
4 & 1 & 3898.687338378537251& 3898.687338454043064& 3898.687338455312334\\ 
4 & 2 & 3898.687338455579702& 3898.687338455344560& 3898.687338455343761\\ 
4 & 3 & 3898.687338455342472& 3898.687338455343756& 3898.687338455343757\\ 
4 & 4 & 3898.687338455343768& 3898.687338455343757& 3898.687338455343757\\ 
4 & 5 & 3898.687338455343757& 3898.687338455343757& 3898.687338455343757\\ 

\hline5 & 0 & 41293.884453984789367& 41293.884401930483936& 41293.884398228979705\\ 
5 & 1 & 41293.884397725045837& 41293.884397813940382& 41293.884397815146458\\ 
5 & 2 & 41293.884397815480301& 41293.884397815172721& 41293.884397815171896\\ 
5 & 3 & 41293.884397815170059& 41293.884397815171891& 41293.884397815171892\\ 
5 & 4 & 41293.884397815171909& 41293.884397815171892& 41293.884397815171892\\ 
5 & 5 & 41293.884397815171892& 41293.884397815171892& 41293.884397815171892\\

\hline
\end{tabular}
\label{tab.Eul1}
\end{table}

\section{Summary}
The series limits of
(\ref{eq.Calphadef})
are $C^{(2)}\approx
38.40676809282179$ \cite[A118582]{sloane},
$C^{(3)}\approx
372.80449187938288$,
and
$C^{(4)}\approx
3898.68733845534376$.

\appendix
\section{The series $\sum_k 1/(k \log^\alpha k)$}
The simpler series
\begin{equation}
D^{(\alpha)}\equiv \sum_{k=2}^\infty \frac{1}{k(\log k)^\alpha}
\end{equation}
can be treated by the same approach. This replaces (\ref{eq.In}) by
\begin{equation}
J_N^{(\alpha)}=\int_{x=N+1/2}^\infty \frac{dx}{x(\log x)^\alpha}
=
\frac{1}{(\alpha-1)[\log (N+1/2)]^{\alpha-1}},
\quad \alpha=2,3,\ldots.
\end{equation}
and (\ref{eq.final}) by
\begin{gather}
D^{(\alpha)}
=
\sum_{k=2}^N
\frac{1}{k(\log k)^\alpha}
+J_N^{(\alpha)}
-
\sum_{s=1}^\infty
\frac{1}{4^s(2s+1)!}
\sum_{k=N+1}^\infty
\frac{d^{2s}}{dx^{2s}}\frac{1}{x(\log x)^\alpha}_{\mid k}
.
\end{gather}
The relevant new set of
derivatives of odd order
starts with
\begin{gather}
\frac{d}{dx}\frac{1}{x (\log x)^\alpha}
=
- \frac{\alpha+\log x}{x^2 (\log x)^{1+\alpha}};
\\
\frac{d^3}{dx^3}\frac{1}{x (\log x)^\alpha}
=
- \frac{\alpha(2+3\alpha+\alpha^2)+6\alpha\log x+(11\alpha+6\log x)(\log x)^2}{x^4 (\log x)^{3+\alpha}}.
\end{gather}
The results are (rounded)
$D^{(2)}\approx  2.10974280123689$ \cite[A115563]{sloane},
$D^{(3)}\approx  2.06588653888414$ \cite[A145419]{sloane},
$D^{(4)}\approx  2.55911974298673$ \cite[A145420]{sloane}.
The value of $D^{(2)}$ is known \cite{BoasAMM84,KreminskiCMJ28,BaxleyMM65,BradenAMM99}.

\section{The series $\sum_k \log k/(k^2-a^2)$}
The Hurwitz Zeta-function is $H(s,a)=\sum_{k\ge 1} 1/(k+a)^s$ \cite{CoffeyJCAM216,CoffeyJCAM21x}, and
its derivative is $H'(s,a)=-\sum_{k \ge 1} \log(k+a)/(k+a)^s$ \cite{ElizaldeJMP34,MillerJCAM100,JohannsonNumAlg69}.
We treat the variant
\begin{equation}
E^{(s)}(a)\equiv \sum_{k}^\infty \frac{\log k}{k^s -a^2}
\label{eq.Edef}
\end{equation}
for integer $s\ge 2$ with the Euler-Maclaurin-Summation technique of the previous chapters; 
the lower limit on $k$ is implied by $k^s \ge 1+a^2 \wedge k>1$.
In lieu of \eqref{eq.In} the integral 
\begin{eqnarray}
K_N^{(s)} & \equiv & \int_{x=N+1/2}^\infty \frac{\log x}{x^s -a^2} dx
\label{eq.Kdef}
\\
&=& 
\frac{2}{s} a^{2/s-2}\log a\int_{z=0}^u \frac{z^{s-2}}{1-z^s} dz 
-a^{2/s-2}\int_{z=0}^u \frac{z^{s-2}\log z}{1-z^s} dz 
\nonumber
\end{eqnarray}
is needed,
where $u\equiv a^{2/s}/(N+1/2)$.

For even $s$ the first integral in \eqref{eq.Kdef} is  \cite[2.146.3]{GR}
\begin{eqnarray*}
\int \frac{z^{s-2}}{1-z^s} dz
&=&
\frac{1}{s}\log \frac{1+z}{1-z}
-\frac{1}{s}\sum_{k=1}^{s/2-1}\cos\frac{2k(s-1)\pi}{s}\ln(1-2z\cos\frac{2k\pi}{s}+z^2)
\\
\nonumber
&&
+\frac{2}{s}\sum_{k=1}^{s/2-1}\sin\frac{2k(s-1)\pi}{s}\arctan \frac{x-\cos\frac{2k\pi}{s}}{\sin\frac{2k\pi}{s}}.
\end{eqnarray*}
For odd $s$ the first integral in \eqref{eq.Kdef} is \cite[2.146.4]{GR}
\begin{eqnarray*}
\int \frac{z^{s-2}}{1-z^s} dz
 &=&
-\frac{1}{s}\log (1-z)
 -\frac{1}{s}\sum_{k=1}^{\lfloor s/2\rfloor}\cos\frac{(2k-1)(s-1)\pi}{s}\ln(1+2z\cos\frac{(2k-1)\pi}{s}+z^2)
\\
\nonumber
&&
 -\frac{2}{s}\sum_{k=1}^{\lfloor s/2\rfloor}\sin\frac{(2k-1)(s-1)\pi}{s}\arctan \frac{z+\cos\frac{(2k-1)\pi}{s}}{\sin\frac{(2k-1)\pi}{s}}
.
\end{eqnarray*}

The second integral in \eqref{eq.Kdef} is a combination of Lerch's Phi functions \cite[\S 1.6]{MO}\cite{KanemitsuAM59}:
\begin{equation}
\int \frac{z^{s-2}}{1-z^s}\log z dz
=
\frac{1}{s} z^{s-1}\left[
\Phi(z^s,1,1-1/s)
\log z 
-
\frac{1}{s}\Phi(z^s,2,1-1/s)
\right].
\end{equation}

For $s=2$ this is rephrased in terms of the Dilogarithm \cite[A.3.2.(3)]{Lewin}\cite{AraiPAMJ4}:
\begin{equation}
K_N^{(2)}
= \frac{\log a}{2a} \log \frac{1+u}{1-u}
- \frac{1}{2a}\left[\Li_2(1-u)+\Li_2(-u)+\log u \log (1+u)-\frac{\pi^2}{6}\right]
.
\end{equation}
The relevant derivatives in \eqref{eq.eul} are superpositions of
\begin{equation}
\bar e(l,j)\equiv \frac{x^l}{(x^s-a^2)^j};
\quad
\hat e(l,j)\equiv \frac{x^l}{(x^s-a^2)^j}\log x,
\end{equation}
which obey the recurrences
\begin{eqnarray}
\frac{d}{dx} \hat e(l,j) &=& \bar e(l-1,j) +l \hat e(l-1,j) - js \hat e(l+s-1,j+1); \\
\frac{d}{dx} \bar e(l,j) &=& l \bar e(l-1,j) -js \bar e(l+s-1,j+1).
\end{eqnarray}
The derivatives of odd order start as
\begin{equation}
\frac{d}{dx} \frac{\log x}{x^s-a^2} =
\frac{d}{dx} \hat e(0,1)=
\bar e(-1,1)-s \hat e(s-1,2);
\end{equation}
\begin{gather}
\frac{d^3}{dx^3} \frac{\log x}{x^s-a^2} 
=
2 \bar e(-3,1)
+3s (2-s) \bar e(s-3,2)
+6s^2 \bar e(2s-3,3)
\\
\nonumber 
-6s^3 \hat e(3s-3,4)
+6s^2(s-1) \hat e(2s-3,3)
-s(s^2-3s+2) \hat e(s-3,2).
\end{gather}
The reference values are
$E^{(2)}(1) \approx 1.023138726427939295535088$ \cite[A340330]{sloane},
$E^{(2)}(3/2) \approx 1.2310517900928141110313101$,
$E^{(2)}(2) \approx 0.9204917248059521124027801$,
$E^{(2)}(i) \approx 0.879146908100343251177878322729$,
where $i$ is the imaginary unit.
\begin{rem}
Expansion of $1/(k^s-a^2)$ into a geometric series of $1/k^s$
and interchange of the order of the two sums
yields representations of $E^{(s)}(a)$ as series of derivatives
of the Riemann Zeta function \cite[A340440,A340484,A340485]{sloane}. Because swapping the orders
of summation is not valid in all cases, these representations are not necessarily
convergent for all $s$ and $a$. 
\end{rem}

\bibliographystyle{amsplain}
\bibliography{all}

\providecommand{\bysame}{\leavevmode\hbox to3em{\hrulefill}\thinspace}
\providecommand{\MR}{\relax\ifhmode\unskip\space\fi MR }
\providecommand{\MRhref}[2]{%
  \href{http://www.ams.org/mathscinet-getitem?mr=#1}{#2}
}
\providecommand{\href}[2]{#2}
\begin{thebibliography}{10}

\bibitem{AS}
Milton Abramowitz and Irene~A. Stegun (eds.), \emph{Handbook of mathematical
  functions}, 9th ed., Dover Publications, New York, 1972. \MR{0167642}

\bibitem{AraiPAMJ4}
Tomihiro Arai, Kalyan Chakraborty, and Jing Ma, \emph{Applications of the
  hurwitz-lerch zeta-function}, Pure Appl. Math. J. \textbf{4} (2015), no.~2-1,
  30--35.

\bibitem{BaxleyMM65}
John~V. Baxley, \emph{Euler's constant, {Taylor}'s formula, and slowly
  converging series}, Math.\ Mag. \textbf{65} (1992), no.~5, 302--313.
  \MR{1191273}

\bibitem{BoasAMM84}
R.~P. {Boas Jr}, \emph{Partial sums of infinite series, and how they grow},
  Amer.\ Math.\ Monthly \textbf{84} (1977), no.~4, 237--258. \MR{0440240}

\bibitem{BradenAMM99}
Bart Braden, \emph{Calculating sums of infinite series}, Amer.\ Math.\ Monthly
  \textbf{99} (1992), no.~7, 649--655. \MR{1176591}

\bibitem{BredeITSF17}
Markus Brede, \emph{A summation formula for convergence acceleration of some
  {Dirichlet} and related series}, Integral Transf. Spec. Func. \textbf{17}
  (2006), no.~10, 703--709. \MR{2252613}

\bibitem{CoffeyJCAM21x}
Mark~W. Coffey, \emph{An efficient algorithm for the hurwitz zeta and related
  functions}, J. Comput. Appl. Math. \textbf{225} (2008), no.~2, 338--346.
  \MR{2494704}

\bibitem{CoffeyJCAM216}
\bysame, \emph{On some series representations of the hurwitz zeta function}, J.
  Comput. Appl. Math. \textbf{216} (2008), no.~1, 297--305. \MR{2421857}

\bibitem{ElizaldeJMP34}
Emili Elizalde, \emph{A simple recurrence for the higher derivatives of the
  hurwitz zeta function}, J. Math. Phys. \textbf{34} (1993), 3222--3226.
  \MR{1224208}

\bibitem{sloane}
O.~E. I.~S. Foundation~Inc., \emph{The {O}n-{L}ine {E}ncyclopedia {O}f
  {I}nteger {S}equences},  (2021), https://oeis.org/. \MR{3822822}

\bibitem{GouldAMM70}
H.~W. Gould and William Squire, \emph{Maclaurin's second formula and its
  generalization}, Amer.\ Math.\ Monthly \textbf{70} (1963), no.~1, 44--52.
  \MR{0146551}

\bibitem{GR}
I.~Gradstein and I.~Ryshik, \emph{Summen-, {P}rodukt- und {I}ntegraltafeln},
  1st ed., Harri Deutsch, Thun, 1981. \MR{0671418}

\bibitem{JohannsonNumAlg69}
Fredrik Johansson, \emph{Rigorous high-precision computation of the hurwitz
  zeta function and its derivatives}, Num. Alg. \textbf{69} (2015), 253--270.
  \MR{3350381}

\bibitem{KanemitsuAM59}
S.~Kanemitsu, M.~Katsurada, and M.~Yoshimoto, \emph{On the hurwitz-lerch
  zeta-function}, Aequation. Mathem. \textbf{59} (2000), no.~1--2, 1--19.
  \MR{1741466}

\bibitem{KreminskiCMJ28}
Rick Kreminski, \emph{Using {Simpson}'s rule to approximate sums of infinite
  series}, College Math.\ J. \textbf{28} (1997), no.~5, 368--376. \MR{1478271}

\bibitem{Lewin}
Leonard Lewin, \emph{Polylogarithms and associated functions}, North Holland,
  1981. \MR{0618278}

\bibitem{MO}
Wilhelm Magnus, Fritz Oberhettinger, and Raj~Pal Soni (eds.), \emph{Formulas
  and theorems for the special functions of mathematical physics}, 3rd ed., Die
  Grundlehren der mathematischen Wissenschaften in Einzeldarstellungen,
  vol.~52, Springer, Berlin, Heidelberg, 1966. \MR{0232968}

\bibitem{MartensenZAMM85}
Erich Martensen, \emph{On the generalized {Euler}-{Maclaurin} formula}, Z.
  Angew. Math. Mech. \textbf{85} (2005), no.~12, 858--863. \MR{2184846}

\bibitem{MillerJCAM100}
Jeff Miller and Victor~S. Adamchik, \emph{Derivatives of the {Hurwitz} zeta
  function for rational arguments}, J. Comput.\ Appl.\ Math. \textbf{100}
  (1998), no.~2, 201--206. \MR{1569109}

\bibitem{VellaINT8}
David~C. Vella, \emph{Explicit formulas for {Bernoulli} and {Euler} numbers},
  Integers: Elec. J.\ Combinat.\ Number Theory \textbf{8} (2008), \#A1.
  \MR{2373085}

\bibitem{ZwillingerCRC}
Daniel Zwillinger (ed.), \emph{{CRC} standard mathematical tables and
  formulae}, 31 ed., Chapman \& Hall/CRC, Boca Raton, FL, 2003, {E}: the lower
  limit in eqs.\ (11)--(14) on page 42 ought be $k=1$, not $k=0$.

\end{thebibliography}

\end{document}